\documentclass[a4paper,11pt]{article}

\usepackage{amsmath,amsthm,amssymb,mathrsfs,latexsym,amsfonts}
\usepackage{graphicx,psfrag,epsfig}
\usepackage{bbm}
\usepackage[english]{babel}
\usepackage[latin1]{inputenc}
\usepackage{pstricks}

\newtheorem{theorem}{Theorem}[section]
\newtheorem{lemma}[theorem]{Lemma}

\newtheorem{corollary}[theorem]{Corollary}
\newtheorem{definition}{Definition\rm}
\newtheorem{conjecture}{Conjecture\rm}

\newcounter{paraga}[section]
\renewcommand{\theparaga}{{\bf\arabic{paraga}.}}
\newcommand{\paraga}{\medskip \addtocounter{paraga}{1} 
\noindent{\theparaga\ } }

\begin{document}

\def\MP{\,{<\hspace{-.5em}\cdot}\,}
\def\SP{\,{>\hspace{-.3em}\cdot}\,}
\def\PM{\,{\cdot\hspace{-.3em}<}\,}
\def\PS{\,{\cdot\hspace{-.3em}>}\,}
\def\EP{\,{=\hspace{-.2em}\cdot}\,}
\def\PP{\,{+\hspace{-.1em}\cdot}\,}
\def\PE{\,{\cdot\hspace{-.2em}=}\,}
\def\N{\mathbb N}
\def\C{\mathbb C}
\def\Q{\mathbb Q}
\def\R{\mathbb R}
\def\T{\mathbb T}
\def\A{\mathbb A}
\def\Z{\mathbb Z}
\def\demi{\frac{1}{2}}

\begin{titlepage}
\author{Abed Bounemoura~\footnote{abed@impa.br, IMPA, Estrada Dona Castorina 110, Rio de Janeiro, Brasil, 22460-320}}
\title{\LARGE{\textbf{Optimal stability and instability for near-linear Hamiltonians}}}
\end{titlepage}

\maketitle

\begin{abstract}
In this paper, we will prove a very general result of stability for perturbations of linear integrable Hamiltonian systems, and we will construct an example of instability showing that both our result and our example are optimal. Moreover, in the same spirit as the notion of KAM stable integrable Hamiltonians, we will introduce a notion of effectively stable integrable Hamiltonians, conjecture a characterization of these Hamiltonians and show that our result prove this conjecture in the linear case. 
\end{abstract}
 
\section{Introduction and results}

\paraga Let $n\geq 2$ be an integer, $\T^n=\R^n/\Z^n$ and $B=B_R$ be an open ball in $\R^n$ of radius $R>1$ with respect to the supremum norm. We shall consider a near-integrable Hamiltonian of the form
\begin{equation*}
\begin{cases} 
H(\theta,I)=h(I)+f(\theta,I) \\
|f| \leq \varepsilon <\!\!<1
\end{cases}
\end{equation*}
where $(\theta,I) \in \T^n \times B$ are action-angle coordinates for the integrable part $h$ and $f$ is a small perturbation in some suitable topology defined by a norm $|\,.\,|$. For simplicity, we shall restrict ourself to the analytic case, that is we assume that $h$ and $f$ are bounded and real-analytic on $D=\T^n \times B$, so that they have holomorphic extensions to some neighbourhood
\[ V_\sigma(D)=\{(\theta,I)\in(\C^n/\Z^n)\times \C^{n} \; | \; |\mathcal{I}(\theta)|<\sigma,\;d(I,B)<\sigma\} \] 
for some $\sigma>0$, where $\mathcal{I}$ denotes the imaginary part and $d$ is the distance associated to the supremum norm. Then the norm of the perturbation $|f|=|f|_\sigma$ is defined by
\[ |f|_\sigma=|f|_{C^0(V_\sigma(D))}=\sup_{z\in V_\sigma(D)}|f(z)|. \] 
In the absence of perturbation, that is when $\varepsilon$ is zero, all solutions $(\theta(t),I(t))$ of the corresponding Hamiltonian system are quasi-periodic and the action variables $I(t)$ are integrals of motion, but in general this is no longer the case after perturbation. Basically, assuming some non-degeneracy condition on the integrable part $h$ and some regularity on $h$ and $f$, KAM theory is concerned with the persistence of many quasi-periodic solutions (\cite{Kol54}), while Nekhoroshev theory deals with the variation of the action components of all solutions (\cite{Nek77}). 

In the realm of KAM theory, one can introduce a notion of ``KAM-stable" integrable Hamiltonian $h$ for which any sufficiently small perturbation possesses a set of positive Lebesgue measure of quasi-periodic solutions closed to the unperturbed ones, and such that the measure $m(\varepsilon)$ of the complement of this set satisfy $\lim_{\varepsilon\rightarrow 0}m(\varepsilon)=0$. These KAM stable Hamiltonians have been characterized: they are exactly Rüssmann non-degenerate Hamiltonians, that is functions $h: B \rightarrow \R$ such that $\nabla h(B)$ is not contained in any hyperplane of $\R^n$. We refer to the nice survey \cite{Sev03} for precise results and references. 

It is tempting to define a notion of ``Nekhoroshev-stable" Hamiltonians and to try to characterize them, but first one has to come up with a precise definition. It is easy to see from the perturbative character of the Hamiltonian that for all solutions $(\theta(t),I(t))$ with initial condition $(\theta_0,I_0)$, one has
\[ \lim_{\varepsilon\rightarrow 0}\left(\sup_{0\leq |t| < \varepsilon^{-1}}|I(t)-I_0|\right)=0. \]
Without further hypothesis on $h$, this cannot be improved. Indeed, following \cite{Nek79} and \cite{Nie06}, if the restriction of $h$ to some affine hyperplane, whose direction is generated by integer vectors, has a non-isolated critical point, then there exist $\delta>0$ and an arbitrarily small perturbation of size $\varepsilon$ such that  
\[ \sup_{0\leq t \leq \varepsilon^{-1}}|I(t)-I_0|\geq \delta.  \] 
A simple example of this type will be given below. This prompts us to introduce the following two definitions.

\begin{definition}
An integrable Hamiltonian $h: B \rightarrow \R$ is rationally steep if its restriction to any affine hyperplane of the form $I_0+\Lambda$, with $I_0\in B$ and $\Lambda$ a linear subspace of $\R^n$ generated by integer vectors, has only isolated critical points.  
\end{definition}

\begin{definition}
An integrable Hamiltonian $h: B \rightarrow \R$ is effectively stable if for any $f : \T^n \times B \rightarrow \R$ with $|f|\leq \varepsilon$, all solutions $(\theta(t),I(t))$ of the Hamiltonian system $H=h+f$ starting at $(\theta_0,I_0)$ satisfy
\[ \lim_{\varepsilon\rightarrow 0}\left(\sup_{0\leq |t|\leq \varepsilon^{-1}}|I(t)-I_0|\right)=0. \] 
\end{definition}

Hence an integrable Hamiltonian is effectively stable if one can ensure that the time of stability $T(\varepsilon)$, that is the maximal time during which the variation $V(\varepsilon)$ of the action of all solutions satisfy $\lim_{\varepsilon\rightarrow 0}V(\varepsilon)=0$, is at least $1/\varepsilon$. Then how large one can choose $T(\varepsilon)$ will depend on more specific properties of the integrable Hamiltonian. 

The following conjecture was essentially made in \cite{BouT}.

\begin{conjecture}\label{conj}
Effectively stable Hamiltonians are exactly rationally steep Hamiltonians.
\end{conjecture}

As we already explained, if a Hamiltonian is effectively stable, then it has to be rationally steep, so that only the converse statement of this conjecture is of interest. In this paper we shall prove this conjecture in the specific case where the integrable Hamiltonian is linear. This is an important case, not only because of its own interest, but also because it is a crucial step in showing the conjecture in its full generality (this point will be briefly discussed at the end of the paper).
  
\paraga So from now on we shall restrict to linear integrable Hamiltonians. Given a vector $\omega\in\R^n\setminus\{0\}$, we let $l(I)=\omega.I$ be the linear Hamiltonian with frequency $\omega$ and we consider  
\begin{equation}\label{H}
\begin{cases}\tag{$*$} 
H(\theta,I)=l(I)+f(\theta,I) \\
|f| \leq \varepsilon <\!\!<1.
\end{cases}
\end{equation}
Without loss of generality, we may assume that the system has been rescaled so that $|\omega|=1$, therefore, reordering the components of $\omega$ if necessary, we can write 
\[ \omega=(1,\alpha_1,\dots,\alpha_{n-1})=(1,\alpha), \quad |\alpha_i|<1, \quad i\in\{1,\dots,n-1\}.\] 
Such ``degenerate" integrable systems are never KAM stable, but the problem of effective stability is of different nature as we will show.

First, one can easily see that $l$ being non-rationally steep is equivalent to $\omega$ being resonant, that is there exists $k\in\Z^n\setminus\{0\}$ such that $k.\omega=0$. In this case, the Hamiltonian 
\[ H(\theta,I)=\omega.I+f(\theta), \quad f(\theta)=-\varepsilon\sin(k.\theta), \]
gives rise to a system which can be easily integrated:
\begin{equation*} 
\left\{  
\begin{array}{ccl}
\dot{\theta} &= & \omega \\
\dot{I} &= & \varepsilon k\cos(k.\theta)  
\end{array}
\right.
\Longrightarrow
\left\{  
\begin{array}{ccl}
\theta(t) & = & \theta_0+t\omega \; [\Z^n]\\
 I(t) &= &I_0+\varepsilon k \cos(k.\theta_0).
\end{array}
\right.
\end{equation*}
So any solution starting at $(\theta_0,I_0)$ with $k.\theta_0=0$ satisfy 
\[ \sup_{0\leq t \leq \varepsilon^{-1}}|I(t)-I_0|=|k|\geq 1. \]
Hence non-rationally steep linear Hamiltonians are indeed not effectively stable, and our conjecture in this case states that linear Hamiltonians with a non-resonant frequency are effectively stable. However, the only case for which results are known is when the frequency satisfies a classical Diophantine condition, and there much stronger stability properties hold true in the sense that the time $T(\varepsilon)$ is at least exponentially large with respect to the inverse of the size of the perturbation. These results will be recalled below. In this paper, we shall prove a more general stability result and we will construct an example showing that it is the best one can obtain.

\paraga Let us state precisely our results. We shall denote by $|\,.\,|_{\Z}$ the distance to the integer lattice $\Z$, that is $|x|_{\Z}=\min_{p\in\Z}|x-p|$ for $x\in\R$. The frequency vector $\omega=(1,\alpha_1,\dots,\alpha_{n-1})=(1,\alpha)$ being non-resonant, the function $\Psi=\Psi_{\omega}$ given by
\[ \Psi(K)=\max\left\{|k.\alpha|_{\Z}^{-1} \; | \; k\in\Z^{n-1}, \, 0<|k|\leq K\right\}, \quad K\in\N^* \]
is well-defined. It is obviously strictly increasing on $\N^*$, hence we can extend it (keeping the same notation) as a strictly increasing continuous function defined on $[1,+\infty)$. Then let us also define two additional functions 
\[ \Lambda(x)=x\Psi(x), \quad \Delta(x)=\Lambda^{-1}(x), \quad x\geq 1,   \]
which are also strictly increasing and continuous.

\begin{theorem}\label{thm1}
Let $H$ be as in~(\ref{H}), with $\omega$ non-resonant. Then there exist positive constants $\varepsilon_0, c, c_1$ and $c_2$ depending only on $n,R,\sigma$ and $\omega$ such that if $\varepsilon\leq\varepsilon_0$, all solutions $(\theta(t),I(t))$ of $H$ with $I_0\in B_{R/2}$ satisfy the estimates
\[ |I(t)-I_0|\leq c_1\delta, \quad |t|\leq \delta\varepsilon^{-1}\exp\left(c_2 \Delta(c\varepsilon^{-1})\right). \]
for any $\left(\Delta\left(c\varepsilon^{-1}\right)\right)^{-1}\leq c_1\delta<R/2$.
\end{theorem}

Depending on the growth of the function $\Delta$, the exponential factor in the time of stability below might not be very large but choosing $\delta=\varepsilon^{b}$ with $b>0$ arbitrarily small, thanks to the factor $\delta\varepsilon^{-1}$ the time of stability is at least $1/\varepsilon$ and this proves our conjecture in the linear case.

\begin{corollary}
For linear integrable Hamiltonians, rationally steep Hamiltonians are effectively stable. 
\end{corollary}

We shall give an elementary proof of Theorem~\ref{thm1} in the case $n=2$, using only one rational approximation and a one-phase averaging, in the same spirit as in \cite{Loc92}. This method of proof can be easily extended for any $n\geq 2$, but in general it fails to give the best result. 

Hence for any $n\geq 2$, we shall use more classical techniques, namely general resonant normal forms as in \cite{Pos93} and \cite{DG96}, and so our proof will not be essentially new. In the special case where the frequency $\omega$ is Diophantine, that is when there exist $\gamma>0$ and $\tau\geq n-1$ such that for all $k\in\Z^{n-1}\setminus\{0\}$, 
$|k.\alpha|_{\Z}\geq \gamma |k|^{-\tau}$, 
then, in Theorem~\ref{thm1}, we can choose
\[ \Psi(x)=\gamma^{-1}x^{\tau}, \quad \Lambda(x)=\gamma^{-1}x^{1+\tau}, \quad \Delta(x)=(\gamma x)^{\frac{1}{1+\tau}}  \]
and our result gives
\[ |I(t)-I_0|\leq c_1\delta, \quad |t|\leq \delta\varepsilon^{-1}\exp\left(c_2 (c\gamma \varepsilon^{-1})^{\frac{1}{1+\tau}}\right). \]
for any $(c^{-1}\gamma^{-1}\varepsilon)^{\frac{1}{1+\tau}}\leq c_1\delta<R/2$. In this case, this is exactly the result obtained in \cite{DG96}. In \cite{Pos93}, there is a similar but less flexible result, since there one can only choose $\delta\sim 1$. Let us also note that this Diophantine case was first considered in \cite{Fas90}, and the result there was even more flexible since one has the freedom to choose $\delta$ equals, up to a multiplicative constant, to $\varepsilon$, but the proof is restricted to the Diophantine case and fails without this assumption.

In any cases, our next theorem shows precisely that the time of stability achieved in Theorem~\ref{thm1} is the best possible.

\begin{theorem}\label{thm2}
For any non-resonant vector $\omega\in\R^n$, there exists a sequence $(f_j)_{j\in\N^*}$ of analytic functions on $V_\sigma(D)$, with $|f_j|_\sigma=\varepsilon_j\rightarrow 0$ when $j\rightarrow +\infty$, such that the system $H_j=l+f_j$ has orbits which satisfy
\[ |I(t)-I_0|=|t|\varepsilon_j\exp\left(-2\sigma \Delta(c\varepsilon_j^{-1})\right). \]
for a constant $c$ depending only on $R,\sigma$ and $\omega$.
\end{theorem}

This result says that for a specific arbitrarily small perturbation and for some solutions of the perturbed system, the inequality of Theorem~\ref{thm1} are in fact equality (up to constants depending only on $n,R,\sigma$ and $\omega$), so that Theorem~\ref{thm1} cannot be improved. Conversely, Theorem~\ref{thm1} implies that the above theorem is the best possible.  

The construction of the example of instability in Theorem~\ref{thm2} is elementary, and it will follow very naturally from our proof of the stability result in the case $n=2$. It is also inspired by an example given in \cite{Sev03}.

Therefore, as far as one is interested in exponential stability (that is, $T(\varepsilon)$ is exponentially large with respect to $\varepsilon^{-1}$), then a polynomial growth of the function $\Psi$, which is nothing but a Diophantine condition, is both sufficient and necessary. To obtain a non-trivial polynomial time of stability such as $\varepsilon^{-r}$, $r>1$, then an exponential growth of the function $\Psi$ is sufficient and necessary.

\paraga The plan of the paper is the following: the proof of the above theorems will be given in the next section, and further comments on the results are given in the last section. Throughout the paper, in order to avoid cumbersome expressions, we will replace positive constants depending only on $n,R,\sigma$ and $\omega$ with a dot. More precisely, an assertion of the form ``there exists a constant $c>0$ depending on the above parameters such that $u < cv$'' will be simply replaced with ``$u\MP v$'', when the context is clear.

\section{Proof of the results}

\paraga We shall start by giving an elementary proof of Theorem~\ref{thm1} in the special case $n=2$.

Given a vector $v=(1,p/q)\in\R^2$, where $p/q$ is a non-zero rational number in its lowest term, we let $l_v$ be the linear Hamiltonian with frequency $v$. The normal form result that we shall need, which is due to Lochak-Neishtadt (\cite{LN92}, see also \cite{Pos99a}), is just a one-phase averaging.  

\begin{lemma}\label{lemlin}
Consider a Hamiltonian $H=l_v+f_v$ defined on $V_{\sigma}(D)$ with $|f_v|_{\sigma}\MP \varepsilon$, and assume that for some $K>1$,
\begin{equation}\label{hyplin}
Kq\varepsilon \MP 1.
\end{equation}
Then there exists an analytic symplectic transformation
\[ \Phi : V_{\sigma/2}(D) \rightarrow V_{\sigma}(D) \]
with $|\Phi-\mathrm{Id}|_{\sigma/2} \MP q\varepsilon$ such that
\[ H'=H\circ\Phi=l_v+g+f_v' \]
with $\{g,l_v\}=0$, and the estimates
\[ |g|_{\sigma/2} \MP \varepsilon, \quad |f_v'|_{\sigma/2} \MP \varepsilon e^{-\cdot K} \]
holds true. 
\end{lemma}

This is a so-called resonant normal form, that is the perturbation $f_v$ has been reduced, up to an exponentially small term $f_v'$, to its resonant part $g$ which satisfies $\{g,l_v\}$. Now a simple observation is that if the denominator $q$ is large, then so is the size of any integer vector which is orthogonal to $v$. In such a case, the resonant term $g$ itself is exponentially small. Let us state this as a lemma which complements the result above.

\begin{lemma}\label{fourier}
Under the previous hypotheses, assume also that $q>K$. Then 
\[ g=\bar{g}+g', \quad \bar{g}=\int_{\T^n}g, \quad |\bar{g}|_{\sigma/2}\MP \varepsilon,  \quad |g'|_{\sigma/4}\MP \varepsilon e^{- \cdot K}. \]
\end{lemma}

\begin{proof}
Expanding $g$ as a Fourier series $g(\theta,I)=\sum_{k\in\Z^n}\hat{g}_k(I)e^{2i\pi k.\theta}$, we can write $g=\bar{g}+g'$ with
\[ g'(\theta,I)=\sum_{k\in\Z^n\setminus\{0\}}\hat{g}_k(I)e^{2i\pi k.\theta}.\]
Now the condition that $\{g,l_v\}=0$ is easily seen to be equivalent to $\hat{g}_k(I)=0$ if $k.v\neq 0$, that is
\[ g'(\theta,I)=\sum_{k.v=0,\;k\neq 0}\hat{g}_k(I)e^{2i\pi k.\theta}. \]
Now take $k=(k_1,k_2)\in\Z^2\setminus\{0\}$ such that $k.v=0$, that is $qk_1+pk_2=0$. First one can see that neither $k_1$ nor $k_2$ is zero, otherwise this would imply that either $p$ or $q$ is zero. Therefore $q$ divides $pk_2$, but as $q$ and $p$ are relatively prime, then $q$ divides $k_2$ and hence $|k|\geq |k_2|\geq q>K$. This means that 
\[ |g'|_{\sigma/2}\leq |g'_{>K}|_{\sigma/2}, \quad g'_{>K}(\theta,I)=\sum_{|k|>K}\hat{g}_k(I)e^{2i\pi k.\theta}. \]
It is now a classical estimate that
\[ |g'_{>K}|_{\sigma/4}\MP |g|_{\sigma/2}e^{- \cdot K} \MP \varepsilon e^{- \cdot K}  \]
with implicit constants depending only on $\sigma$.
\end{proof}

Part of the arguments given above works only for $n=2$. However, for any $n\geq 2$, some other simple arguments along these lines can be used (but as we already said, this does not give the best result in general). Now we can prove Theorem~\ref{thm1} for $n=2$.

\begin{proof}[Proof of Theorem~\ref{thm1}, $n=2$]
Here $\omega=(1,\alpha)\in\R^2$ with $|\alpha|<1$. The size of the perturbation $\varepsilon>0$ being given, we choose $K$ such that
\[ K\Psi(K)\EP\varepsilon^{-1}, \quad K=\Delta(\cdot\varepsilon^{-1}) \]
for some implicit constant to be chosen below. First, assuming $\varepsilon$ is sufficiently small, that is $\varepsilon\MP 1$, we can ensure that $K>1$ and we can apply Dirichlet's box principle to approximate $\alpha$ by a rational number with denominator $q<\Psi(K)$: we obtain a vector $v=(1,p/q)\in\R^2$, $p\in\Z\setminus\{0\}$, such that 
\[ |\omega-v|\leq q^{-1}\Psi(K)^{-1}, \quad 1<q<\Psi(K). \]
Moreover, from the definition of $\Psi$,
\[ \Psi(q)^{-1}\leq|q\alpha-p|<\Psi(K)^{-1}, \]  
so that $q>K$. Hence
\[ |\omega-v|\leq K^{-1}\Psi(K)^{-1}\EP\varepsilon. \]
Now our Hamiltonian $H$ can be written as
\[ H=l+f=l_v+f_v, \quad f_v=l-l_v+f, \]
with $|f_v|_\sigma\MP\varepsilon$. Moreover, since $q<\Psi(K)$, then $q\varepsilon \MP K^{-1}$, and choosing properly our implicit constant in the definition of $K$, we can ensure that condition~(\ref{hyplin}) is met and hence we can apply Lemma~\ref{lemlin}: there exists an analytic symplectic transformation
\[ \Phi : V_{\sigma/2}(D) \rightarrow V_{\sigma}(D) \]
with $|\Phi-\mathrm{Id}|_{\sigma/2}\MP q\varepsilon \MP K^{-1}$ such that
\[ H'=H\circ\Phi=l_v+g+f_v' \]
with $\{g,l_v\}=0$, and the estimates
\[ |g|_{\sigma/2} \MP \varepsilon, \quad |f_v'|_{\sigma/2} \MP \varepsilon e^{-\cdot K} \]
holds true. Moreover, since $q>K$, Lemma~\ref{fourier} can also be applied and therefore 
\[ H'=H\circ\Phi=l_v+\bar{g}+g'+f_v'=l_v+\bar{g}+\tilde{f} \]
with $\bar{g}$ integrable and 
\[ |\tilde{f}|_{\sigma/4} \leq |g'|_{\sigma/4}+|f_v'|_{\sigma/4} \MP \varepsilon e^{-\cdot K}.\]
Now let us write $(\theta,I)=\Phi(\theta',I')$, and take any $\delta$ such that $K^{-1}\MP\delta<R/2$. Since $l_v+\bar{g}$ is integrable, the mean value theorem gives
\[ |I'(t)-I'_0|\MP \delta, \quad |t|\leq \delta\varepsilon^{-1} e^{\cdot K}. \]
Now coming back to the original coordinates and using $|\Phi-\mathrm{Id}|_{\sigma/2} \MP K^{-1}$, this gives
\[ |I(t)-I_0|\MP \delta, \quad |t|\leq \delta\varepsilon^{-1} e^{\cdot K}, \]
which, recalling that $K=\Delta(\cdot\varepsilon^{-1})$, is our statement. This completes the proof.
\end{proof}

\paraga Now for any number of degrees of freedom, the strategy we explained above can be easily extended and a result of stability can be obtained. However, it requires to compare simultaneous and linear Diophantine approximation, and it seems to us that, unless the frequency is badly approximable, it cannot give the best result from a quantitative point of view.

Therefore we shall use different arguments to prove Theorem~\ref{thm1} for any $n\geq 2$, following the ``dual" approach. Given any sub-lattice $\Lambda$ of $\Z^n$, we shall say that a function $g$ defined on $V_\sigma(D)$ is $\Lambda$-resonant if its Fourier expansion is of the form
\[ g(\theta,I)=\sum_{k\in \Lambda}\hat{g}_k(I)e^{2i\pi k.\theta}, \quad (\theta,I)\in V_\sigma(D). \]
Then, given any $K\geq 1$ and any $\lambda>0$, we shall say that a vector $w\in\R^n$ is $(\lambda,K)$-non resonant modulo $\Lambda$ if
\[ \forall k\in\Z^n\setminus\Lambda, \quad |k|\leq K, \quad |k.w|\geq \lambda.  \]
Finally, given such a vector $w$, let us write $l_w(I)=w.I$. The normal form result that we shall use here is due to Delshams-Gutiérrez (\cite{DG96}).

\begin{lemma}\label{lemlin2}
Consider a Hamiltonian $H=l_w+f$ defined on $V_{\sigma}(D)$, with $|f|_{\sigma}\MP \varepsilon$ and assume that 
\begin{equation}\label{hyplin2}
K\lambda^{-1}\varepsilon \MP 1.
\end{equation}
Then there exists an analytic symplectic transformation
\[ \Phi : V_{\sigma/2}(D) \rightarrow V_{\sigma}(D) \]
with $|\Phi-\mathrm{Id}|_{\sigma/2} \MP \lambda^{-1}\varepsilon$ such that
\[ H'=H\circ\Phi=l_w+g+f' \]
where $g$ is $\Lambda$-resonant, and with the estimates
\[ |g|_{\sigma/2} \MP \varepsilon, \quad |f'|_{\sigma/2} \MP \varepsilon e^{-\cdot K}. \]
\end{lemma}

The above normal form result strictly contains Lemma~\ref{lemlin} (the latter corresponds to the case where $\Lambda$ has rank $n-1$, and we can choose $\lambda=q^{-1}$ for any $K\geq 1$), as a consequence its proof is more involved. A similar normal form result is contained in \cite{Pos93}, but for our purpose it is weaker since the distance to the identity of the normalizing transformation $\Phi$ is bigger (it is only of order $K\lambda^{-1}\varepsilon$).

The proof of Theorem~\ref{thm1} is now straightforward.

\begin{proof}[Proof of Theorem~\ref{thm1}]
Here $\omega=(1,\alpha_1,\dots,\alpha_{n-1})$, and this frequency is by definition $(\Psi(K)^{-1},K)$-non-resonant modulo $\{0\}$. The size of the perturbation $\varepsilon>0$ being given, we choose $K$ such that
\[ K\Psi(K)\EP\varepsilon^{-1}, \quad K=\Delta(\cdot\varepsilon^{-1}) \]
with a suitable implicit constant so that the condition~(\ref{hyplin2}) is fulfilled. Then the requirement that $K\geq 1$ gives a threshold $\varepsilon\MP 1$. Therefore we can apply Lemma~\ref{lemlin2}: there exists an analytic symplectic transformation
\[ \Phi : V_{\sigma/2}(D) \rightarrow V_{\sigma}(D) \]
with $|\Phi-\mathrm{Id}|_{\sigma/2} \MP \Psi(K)\varepsilon \EP K^{-1}$ such that
\[ H'=H\circ\Phi=l+g+f' \]
where $g$ is $\{0\}$-resonant, that is $g=\bar{g}$, and with the estimates
\[ |g|_{\sigma/2} \MP \varepsilon, \quad |f|_{\sigma/2} \MP \varepsilon e^{-\cdot K}. \]
Now, as before, let us write $(\theta,I)=\Phi(\theta',I')$, and take any $\delta$ such that $K^{-1}\MP\delta<R/2$. Since $l+\bar{g}$ is integrable, the mean value theorem gives
\[ |I'(t)-I'_0|\MP \delta, \quad |t|\leq \delta\varepsilon^{-1} e^{\cdot K}. \]
Now coming back to the original coordinates and using $|\Phi-\mathrm{Id}|_{\sigma/2} \MP K^{-1}$, this gives
\[ |I(t)-I_0|\MP \delta, \quad |t|\leq \delta\varepsilon^{-1} e^{\cdot K} \]
which ends the proof, since $K=\Delta(\cdot\varepsilon^{-1})$.
\end{proof}

\paraga Now let show that our stability result is optimal by proving Theorem~\ref{thm2}. 

\begin{proof}[Proof of Theorem~\ref{thm2}]
Given $\omega=(1,\alpha_1,\dots,\alpha_{n-1})$, let us pick one component, say $\alpha_1$, and let us denote by $(p_j/q_j)_{j\in\N}$ the sequence of the convergents of $\alpha_1$. The vector $\omega$ being non-resonant, $\alpha_1$ is irrational and this means that the sequence $(q_j)_{j\in\N}$ is strictly increasing. From the classical estimates
\[ (q_j+q_{j+1})^{-1}<|q_j\alpha_1-p_j|<q_{j+1}^{-1}, \quad j\in\N, \]
we deduce that $q_{j+1}<\Psi(q_j)<q_j+q_{j+1}<2q_{j+1}$. Now the perturbation $f_j$ will be of the form
\[ f_j(\theta,I)=f_j^1(I)+f_j^2(\theta), \quad (\theta,I)\in \T^n \times B. \]
First, we choose $f_j^1(I)=v_j.I-\omega.I$, where $v_j=(1,p_j/q_j,\alpha_2,\dots,\alpha_{n-1})$. As
\[ |\alpha_1-p_jq_j^{-1}|<(q_jq_{j+1})^{-1}<2(q_j\Psi(q_j))^{-1}, \]
if we set 
\[ \varepsilon_j=c(q_j\Psi(q_j))^{-1}, \quad q_j=\Delta(c\varepsilon_j^{-1})\]
with a suitable constant $c$ depending on $\alpha_1,R$ and $\sigma$, we obtain $|f_j^1|_\sigma<\varepsilon_j/2$. Then, if we define $k_j=(p_j,-q_j,0,\dots,0)\in \Z^n$, we choose
\[ f_j^2(\theta)=\varepsilon_j\mu_j \cos(k_j.\theta), \quad \mu_j=q_j^{-1}\exp(-2\sigma q_j). \]
For $\theta\in \C^n$ with $|\mathcal{I}(\theta)|\leq\sigma$, we have
\[ |\cos(k_j.\theta)|\leq \exp(2\sigma|k_j|)=\exp(2\sigma q_j)\leq q_j/2 \exp(2\sigma q_j) \]
and therefore $|f_j^2|_\sigma \leq \varepsilon_j/2$ as well. Hence $|f_j|_{\sigma}\leq\varepsilon_j$, and $\varepsilon_j \rightarrow 0$ when $j\rightarrow +\infty$. Now we can write the Hamiltonian
\[ H_j(\theta,I)=h(I)+f_j(\theta,I)=v_j.I+\varepsilon_jq_j^{-1}\exp(-2\sigma q_j) \cos(k_j.\theta) \]
and as $k_j.v_j=0$, the associated system is easily integrated:
\begin{equation*} 
\left\{  
\begin{array}{ccl}
\theta(t) & = & \theta_0+tv_j \quad [\Z^n]\\
 I(t) &= &I_0-tk_jq_j^{-1}\varepsilon_j\exp(-2\sigma q_j)\cos(k_j.\theta_0).
\end{array}
\right. 
\end{equation*}
Choosing any solution with initial condition $(\theta_0,I_0)$ satisfying $k_j.\theta_0=0$, $\cos(k_j.\theta_0)=1$ and using the fact that $|k_j|q_j^{-1}=1$, we obtain
\[ |I(t)-I_0|=|t|\varepsilon_j\exp(-2\sigma q_j). \]
Recalling that $q_j=\Delta(c\varepsilon_j^{-1})$ this gives
\[ |I(t)-I_0|=|t|\varepsilon_j\exp\left(-2\sigma \Delta(c\varepsilon_j^{-1})\right) \]
and this concludes the proof.
\end{proof}

\section{Further comments}\label{comments}

Let us conclude this paper by several remarks.

\paraga First, let us mention that a perturbation of a linear Hamiltonian system, as we have considered, also describes the dynamics in the neighbourhood of an invariant, linearly stable, isotropic and reducible torus carrying a quasi-periodic motion. Our Theorem~\ref{thm1} directly applies to the case where the torus is of full dimension (that is, when the torus is Lagrangian) and gives a non-trivial result of stability. In the case of intermediate dimensions, our result should apply also but there one has to work not only with angle-action coordinates but also with Cartesian coordinates. Anyway, in the case where the torus has dimension zero, that is for an elliptic fixed point, Theorem~\ref{thm1} is useless since in this situation the perturbation has a peculiar form and classical Birkhoff normal form estimates apply and give a much better result of stability. Now Theorem~\ref{thm2} also applies in this setting, except for elliptic fixed point since in this case there is no proper angle-action coordinates to work with (and in fact, for $n=2$, Diophantine and even Brjuno elliptic fixed points are always Lyapounov stable, see \cite{Rus02} and \cite{FK09}). When the dimension of the torus is at least one, our result applies and in the Diophantine case, without any further assumptions, results of exponential stability are the best one can expect. Assuming further conditions on the Birkhoff invariants, it was shown in \cite{MG95} (see also \cite{Bou09} for more general results) that super-exponential stability hold true. Here our result shows that some further conditions are indeed necessary.

\paraga Then, let us discuss the case where the system is assumed to be less regular than analytic. Our proof of Theorem~\ref{thm1} in the case $n=2$ only uses a one-phase averaging. Such a result being available in the Gevrey category (\cite{MS02}) and in the finitely differentiable category (\cite{Bou10}), the proof immediately extends in this case (of course, in finite differentiability, only polynomial stability hold true). Now for any $n\geq 2$, we have used a more elaborate normal form which is not known if the system is not analytic. However, such a normal form can certainly be proved using smoothing techniques so that Theorem~\ref{thm1} should hold in any regularity. Obviously, the proof of Theorem~\ref{thm2} is independent of the regularity and the extension is straightforward.  

\paraga Finally, let us explain how our result can be used to prove in general that effectively stable Hamiltonians are exactly rationally steep Hamiltonians. As we have explained, our result deals with the case where the Hamiltonian is linear. In the non-linear case, first one has to use not only integrable normal forms but also more general resonant forms, which are also well-known. Then, the condition of being rationally steep roughly says that resonances ``do not accumulate" so that any solution, when evolving (of course, if they do not evolve there is nothing to prove), will necessarily encounters non-resonant points, arbitrarily close and in any direction, and therefore the inductive scheme introduced in \cite{Nie07} (see also \cite{BN09} and \cite{Bou11} for refinements), together with some other geometric arguments, should give the desired result.   

\bigskip

{\it Acknowledgments.} Thanks to Jean-Christophe Yoccoz, Bassam Fayad, Stéphane Fischler, Pierre Lochak and special thanks to Laurent Niederman and Jean-Pierre Marco. This work has been done during a stay at Penn State University, the author is grateful to Vadim Kaloshin for his very kind invitation. Finally, the author thanks the CNPq for financial support. 

\addcontentsline{toc}{section}{References}
\bibliographystyle{amsalpha}
\bibliography{Lineaire3}

\end{document}